\documentclass[a4paper,10pt]{amsart}

\usepackage{amssymb,amsmath,amsfonts,amsthm,mathtools}
\usepackage{latexsym,mathrsfs,pb-diagram}
\usepackage[pdftex]{graphicx}
\usepackage[all]{xy}
\usepackage[hmargin=2.5cm,vmargin=3.5cm]{geometry}
\usepackage[plainpages=false,colorlinks,pdfpagelabels]{hyperref}
\hypersetup{
  urlcolor=black,
  citecolor=green,
  linkcolor=blue
}

\numberwithin{equation}{section}

\theoremstyle{definition}
\newtheorem{Def}{Definition}[section]

\theoremstyle{remark}

\theoremstyle{plain}

\newtheorem{Lem}[Def]{Lemma}

\newcommand{\dfn}{\mathrel{\dot{=}}}
\newcommand{\impl}{\Rightarrow}

\newcommand{\st}{ \ ; \ }
\newcommand{\rarr}{\rightarrow}
\newcommand{\sset}{\subset}

\newcommand{\TR}[5]{\begin{array}{c c c c c}
    {#1} & : & {#3} & \longrightarrow & {#5}\\
    & & {#2} & \longmapsto & {#4}
  \end{array}
}

\newcommand{\eset}{\emptyset}

\newcommand{\R}{\mathbb{R}}

\newcommand{\C}{\mathbb{C}}

\newcommand{\del}{\partial}

\DeclareMathOperator{\T}{T}
\newcommand{\dd}{\mathrm{d}}

\newcommand{\D}{\mathcal{D}}
\newcommand{\E}{\mathcal{E}}
\newcommand{\OO}{\mathcal{O}}

\DeclareMathOperator{\supp}{\mathrm{supp}}

\author{Gabriel Ara{\'u}jo}
\date{\today}
\title[Missing proofs]{Regularity and solvability of linear differential operators in Gevrey spaces: omitted proofs}
\begin{document}
\maketitle

This is an addendum to~\cite{araujo18} which aims to provide the proofs of some results in that paper (Theorem~7.5 and Proposition~9.15) which were removed from its final version. The reason for such omission is that these proofs follow quite closely others already present in the literature, with minor modifications. I make them publicly available for the sake of completeness.

\section{Proof of Theorem~7.5}

Our main reference here is H{\"o}rmander~\cite{hormander_alpdo2}. We will start by proving analogous versions of several auxiliary lemmas used in his book, which we did not find in the literature (especially the ones not covered by Bj{\"o}rck~\cite{bjorck66}). Although the proofs of these lemmas are very much like their counterparts in~\cite{hormander_alpdo2}, we chose to present them here for the sake of completeness. We will, however, make free use of the results already proven in~\cite{bjorck66}. 

For the first result in this section (which is an adaptation of~\cite[Theorem~10.1.5]{hormander_alpdo2}) we recall that for each $k \in \mathscr{K}_\omega$ one defines, in accordance with~\cite{bjorck66} and~\cite{hormander_alpdo2},
\begin{align*}
  M_k(\xi) &\dfn \sup_\eta \frac{k(\xi + \eta)}{k(\eta)} 
\end{align*}
Also, for $\lambda > 0$ let $\mathscr{K}_\omega^\lambda$ stand for the set of functions $k \in \mathscr{K}_\omega$ such that
\begin{align}
  k(\xi + \eta) &\leq e^{\lambda |\xi|^\frac{1}{\sigma}} k(\eta), \quad \forall \xi, \eta \in \R^n, \label{eq:K_lambda}
\end{align}
so $\mathscr{K}_\omega$ is exactly the union of all $\mathscr{K}_\omega^\lambda$.

\begin{Lem} \label{lem:kdelta0} For each $\lambda > 0$, each $k \in \mathscr{K}_\omega^\lambda$ and each $\delta > 0$ there exist $k_\delta \in \mathscr{K}_\omega^\lambda$ and $C_\delta > 0$ such that, for every $\xi \in \R^n$ one has
  \begin{enumerate}
  \item $1 \leq k_\delta(\xi) / k(\xi) \leq C_\delta$ and
  \item $1 \leq M_{k_\delta}(\xi) \leq e^{\delta|\xi|}$.
  \end{enumerate}
  \begin{proof} For $\delta > 0$ let
    \begin{align*}
      k_\delta(\xi) &\dfn \sup_\eta e^{-\delta |\eta|} k(\xi - \eta), \quad \xi \in \R^n,
    \end{align*}
    which defines and element of $\mathscr{K}_\omega^\lambda$. Indeed, for $\xi, \xi' \in \R^n$ we have
    \begin{align*}
      k_\delta(\xi + \xi') = \sup_\eta e^{-\delta |\eta|} k(\xi + \xi' - \eta) \leq \sup_\eta e^{-\delta |\eta|} e^{\lambda |\xi'|^\frac{1}{\sigma}}k(\xi - \eta) = e^{\lambda |\xi'|^\frac{1}{\sigma}}  k_\delta(\xi).
    \end{align*}
    Notice that
    \begin{align*}
      k(\xi) \leq k_\delta(\xi) = \sup_\eta e^{-\delta |\eta|} k(\xi - \eta) \leq \sup_\eta e^{-\delta |\eta|} e^{\lambda |\eta|^\frac{1}{\sigma}}  k(\xi) = k(\xi) \sup_\eta e^{\lambda |\eta|^\frac{1}{\sigma} - \delta |\eta|}
    \end{align*}
    where the constant on the far right (call it $C_\delta$) is finite, proving the first statement. A change of variables allows us to write, for $\xi \in \R^n$,
    \begin{align*}
      k_\delta(\xi) &= \sup_\eta e^{-\delta |\xi - \eta|} k(\eta)
    \end{align*}
    and so we have
    \begin{align*}
      k_\delta(\xi + \xi') = \sup_\eta e^{-\delta |\xi + \xi' - \eta|} k(\eta) \leq e^{\delta|\xi'|} \sup_\eta e^{-\delta |\xi - \eta|} k(\eta) = e^{\delta|\xi'|} k_\delta(\xi)
    \end{align*}
    thus implying that $M_{k_\delta}(\xi') \leq e^{\delta|\xi'|}$. 
  \end{proof}
\end{Lem}

Now we present a version of~\cite[Lemma~13.3.1]{hormander_alpdo2}.
\begin{Lem} \label{lem:kdelta} Let $k \in \mathscr{K}_\omega$ and, for each $\delta > 0$, let $k_\delta \in \mathscr{K}_\omega$ be as in Lemma~\ref{lem:kdelta0}. Then for each $\phi \in \D_\omega(\R^n)$ there exists $\delta_0 > 0$ such that
  \begin{align*}
    \| \phi u \|_{p, k_\delta} &\leq 2 \| \phi \|_{1,1} \| u \|_{p, k_\delta}
  \end{align*}
  for every $0 < \delta < \delta_0$ and every $u \in \mathcal{B}_{p, k_\delta} = \mathcal{B}_{p, k}$.
  \begin{proof} From~\cite[Theorem~2.2.7]{bjorck66} we have, for every $\delta > 0$,
    \begin{align*}
      \| \phi u \|_{p, k_\delta} &\leq \| \phi \|_{1, M_{k_\delta}} \| u \|_{p, k_\delta}
    \end{align*}
    so it is enough to prove the existence of a $\delta_0 > 0$ such that
    \begin{align*}
      \| \phi \|_{1, M_{k_\delta}} &\leq 2 \| \phi \|_{1,1}
    \end{align*}
    for every $0 < \delta < \delta_0$. But from the definition of the norms we have 
    \begin{align*}
      \| \phi \|_{1, M_{k_\delta}} = \frac{1}{(2 \pi)^n} \int M_{k_\delta}(\xi) \ |\hat{\phi}(\xi)| \ \dd \xi \to \frac{1}{(2 \pi)^n} \int |\hat{\phi}(\xi)| \ \dd \xi = \| \phi \|_{1,1}
    \end{align*}
    because $M_{k_\delta} \to 1$ uniformly on compact set as $\delta \to 0^+$: this follows immediately from Lemma~\ref{lem:kdelta0}, which also implies that $\mathcal{B}_{p, k_\delta}$ and $\mathcal{B}_{p, k}$ are the same as topological vector spaces, since their norms are equivalent.
  \end{proof}
\end{Lem}

Now we proceed with the proof of Theorem~7.5 from~\cite{araujo18}. We shall not reproduce its statement here. Due to~\cite[Lemma~13.1.2]{hormander_alpdo2} there exist operators with constant coefficients $P_1(D), \ldots, P_r(D)$ and functions $c_0, c_1, \ldots, c_r \in C^\infty(\Omega)$, that are uniquely determined by the following properties:
\begin{itemize}
\item $P_j \prec P_0$ for every $j \in \{1, \ldots, r\}$;
\item $c_j(x_0) = 0$ for every $j \in \{0, \ldots, r\}$;
\item and, in $\Omega$,
  \begin{align*}
    P(x,D) &= P_0(D) + \sum_{j = 1}^r c_j(x) P_j(D).
  \end{align*}
\end{itemize}
Since we are also assuming that the coefficients of $P(x,D)$ belong to $G^{\sigma_0}(\Omega)$ one can actually show that $c_0, c_1, \ldots, c_r \in G^{\sigma_0}(\Omega)$. For every $\epsilon > 0$, define
\begin{align*}
  X_\epsilon &\dfn \{ x \in \R^n \st |x - x_0| < \epsilon \}
\end{align*}
and select $\epsilon_0 > 0$ such that $X_{\epsilon_0} \sset \Omega$. Let $\chi \in G^{\sigma_0}_c(\R^n)$ be equal to $1$ in a neighborhood of $\{ x \in \R^n \st |x| \leq 2 \epsilon_0 \}$ and 
\begin{align*}
  E_0 &\in B_{\infty, \tilde{P}_0}^{\mathrm{loc}}(\R^n)
\end{align*}
be a fundamental solution of $P_0(D)$, and define
\begin{align*}
  F_0 \dfn \chi E_0 \in B_{\infty, \tilde{P_0}}.
\end{align*}
If $g \in \E'_\omega(\R^n)$ has its support in $X_{\epsilon_0}$ then $F_0 * g = E_0 * g$ in $X_{\epsilon_0}$, hence
\begin{align*}
  P_0(D) (F_0 * g) = F_0 * P_0(D)g = g \ \text{in $X_{\epsilon_0}$}.
\end{align*}
Now let $\psi \in G^{\sigma_0}_c(\R^n)$ be such that
\begin{align*}
  \psi &= 1 \ \text{in $\{ x \in \R^n \st |x| \leq 1 \}$} \\
  \psi &= 0 \ \text{in $\{ x \in \R^n \st |x| > 2 \}$}
\end{align*}
and define $\psi_\epsilon(x) \dfn \psi((x - x_0) / \epsilon)$. We claim the existence of $0 < \epsilon_1 < \epsilon_0 / 2$ such that for each $0 < \epsilon < \epsilon_1$ and each $f \in \E'_\omega(\R^n)$ the equation
\begin{align}
  g + \sum_{j = 0}^r \psi_\epsilon c_j P_j(D) (F_0 * g) &= \psi_\epsilon f \label{eq:fix_point}
\end{align}
has a unique solution $g \in \E'_\omega(\R^n)$. Proceeding as in~\cite{hormander_alpdo2}, we provisionally assume this claim and define the operator $E$ as
\begin{align*}
  Ef &\dfn F_0 * g
\end{align*}
where $g\in \E'_\omega(\R^n)$ is the unique solution of~\eqref{eq:fix_point}, which yields a linear map $E: \E'_\omega(\R^n) \rarr \E'_\omega(\R^n)$: we will prove that if $\epsilon > 0$ is small enough then this operator has the properties described in the statement above. 

First, since $\supp \psi_\epsilon \sset X_{\epsilon_0}$  equation~\eqref{eq:fix_point} implies that $\supp g \sset X_{\epsilon_0}$, so in $X_\epsilon$ 
\begin{align*}
  P(x,D) E f &= P(x, D) (F_0 * g) \\
  &= P_0(D) (F_0 * g) + \sum_{j = 0}^r c_j P_j(D) (F_0 * g) \\
  &= g + \sum_{j = 0}^r \psi_\epsilon c_j P_j(D) (F_0 * g) \\
  &= \psi_\epsilon f \\
  &= f
\end{align*}
thus proving the first property claimed. Second, let $u \in \E'_\omega(\R^n)$ be such that $\supp u \sset X_\epsilon$ and $f \dfn P(x,D)u$: putting $g \dfn P_0(D)u$ in the left-hand side of~\eqref{eq:fix_point} we get
\begin{align*}
  g + \sum_{j = 0}^r \psi_\epsilon c_j P_j(D) (F_0 * g) &= P_0(D)u + \sum_{j = 0}^r \psi_\epsilon c_j P_j(D) (F_0 * P_0(D)u) \\
  &= P_0(D)u + \sum_{j = 0}^r \psi_\epsilon c_j P_j(D) u \\
  &= P(x,D) u \\
  &= f \\
  &= \psi_\epsilon f
\end{align*}
that is, $g$ solves equation~\eqref{eq:fix_point}, and by uniqueness we have
\begin{align*}
  Ef = F_0 * g = F_0 * P_0(D)u = u.
\end{align*}
This proves the second property of $E$.

The last property of $E$ -- the estimate between norms -- will follow from the proof of our claim about existence and uniqueness of solutions of equation~\eqref{eq:fix_point}, so now we proceed in that direction. For every $\epsilon > 0$ we define a linear map $A_\epsilon: \D'_\omega(\R^n) \rarr \D'_\omega(\R^n)$ by the expression
\begin{align*}
  A_\epsilon g &\dfn \sum_{j = 0}^r \psi_\epsilon c_j P_j(D) (F_0 * g)
\end{align*}
which is well-defined for every $g \in \D'_\omega(\R^n)$, for $F_0$ is compactly supported. Let $k \in \mathscr{K}_\omega$ and, for $\delta > 0$, let $k_\delta \in \mathscr{K}_\omega$ as in Lemma~\ref{lem:kdelta} (in which case $\mathcal{B}_{p, k_\delta} = \mathcal{B}_{p, k}$, with equivalent defining norms): according to it, there exists $\delta_0 > 0$ such that if $0 < \delta < \delta_0$ one has
\begin{align*}
  \| A_\epsilon g \|_{p, k_\delta} &\leq \sum_{j = 0}^r \| \psi_\epsilon c_j P_j(D) (F_0 * g) \|_{p, k_\delta} \\
  &\leq 2 \sum_{j = 0}^r \| \psi_\epsilon c_j \|_{1,1} \| P_j(D) (F_0 * g) \|_{p, k_\delta}
\end{align*}
as long as $P_j(D) (F_0 * g) \in \mathcal{B}_{p, k}$ (recall that $\psi_\epsilon c_j \in \D_\omega(\R^n)$ for every $j \in \{ 0, \ldots, r \}$ according to~\cite[Lemma~7.4]{araujo18}. Now, since $P_j \prec P_0$ and $F_0 \in B_{\infty, \tilde{P}_0}$ there are constants $C_1, C_2 > 0$ such that
\begin{align*}
  |P_j(\xi)| |\hat{F}_0(\xi)| \leq |\tilde{P}_j(\xi)| |\hat{F}_0(\xi)| \leq C_1 |\tilde{P}_0(\xi)| |\hat{F}_0(\xi)| \leq C_1 C_2
\end{align*}
for every $\xi \in \R^n$, so if we define $C \dfn C_1 C_2 > 0$ we have that
\begin{align*}
  \| P_j(D) (F_0 * g) \|_{p, k_\delta} = \| k_\delta \ P_j \ \hat{F}_0 \ \hat{g} \|_{L^p} \leq C \| k_\delta \ \hat{g} \|_{L^p} = C \| g \|_{p, k_\delta}
\end{align*}
for every $g \in \mathcal{B}_{p, k}$: therefore
\begin{align*}
  \| A_\epsilon g \|_{p, k_\delta} &\leq 2C \sum_{j = 0}^r \| \psi_\epsilon c_j \|_{1,1} \| g \|_{p, k_\delta}
\end{align*}
and thus $A_\epsilon: \mathcal{B}_{p, k} \rarr \mathcal{B}_{p, k}$ continuously. Now~\cite[Lemma~13.3.2]{hormander_alpdo2} allows us to choose $0 < \epsilon_1 < \epsilon_0 / 2$ such that
\begin{align*}
  \sum_{j = 0}^r \| \psi_\epsilon c_j \|_{1,1} &\leq \frac{1}{4C}
\end{align*}
for every $0 < \epsilon < \epsilon_1$. We stress that such a choice is independent of $k$, and hence
\begin{align}
  \| A_\epsilon g \|_{p, k_\delta} &\leq \frac{1}{2} \| g \|_{p, k_\delta} \label{eq:est_Aeps}
\end{align}
for every $g \in \mathcal{B}_{p, k}$. We conclude that $I + A_\epsilon: \mathcal{B}_{p, k} \rarr \mathcal{B}_{p, k}$ is invertible, which means that equation~\eqref{eq:fix_point} has a unique solution $g \in \mathcal{B}_{p, k}$ whenever $f \in \mathcal{B}_{p, k}$, which must have compact support for reasons already mentioned. We need one more lemma to finish this argument.

\begin{Lem} \label{lem:compact_supp_bjorck} Let $1 \leq p \leq \infty$. Every $u \in \E'_\omega(\R^n)$ belongs to $\mathcal{B}_{p,k}$ for some $k \in \mathscr{K}_\omega$.
  \begin{proof}[Proof of Lemma~\ref{lem:compact_supp_bjorck}] For $u \in \E'_\omega(\R^n)$, \cite[Theorem~1.8.14]{bjorck66} ensures, among other things, the existence of constants $\lambda \in \R$ and $C > 0$ such that
    \begin{align*}
      |\hat{u}(\xi)| &\leq C e^{\lambda |\xi|^{\frac{1}{\sigma}}}, \quad \forall \xi \in \R^n.
    \end{align*}
    Of course we can assume $\lambda > 0$, so $k(\xi) \dfn e^{-2\lambda |\xi|^{\frac{1}{\sigma}}}$ defines an element of $\mathscr{K}_\omega$ and 
    \begin{align*}
      k(\xi) |\hat{u}(\xi)| &\leq C e^{-\lambda |\xi|^{\frac{1}{\sigma}}}, \quad \forall \xi \in \R^n,
    \end{align*}
    so $k \hat{u} \in L^p(\R^n)$ (i.e.\ $u \in  \mathcal{B}_{p,k}$) no matter what $p$ is.
  \end{proof}
\end{Lem}

Now we turn back to the deduction of estimate~(7.2) in the statement of the theorem (see~\cite{araujo18}). Let $f \in \E'_\omega(\R^n) \cap \mathcal{B}_{p,k}$ and take $g \in \E'_\omega(\R^n) \cap \mathcal{B}_{p,k}$ the unique solution of~\eqref{eq:fix_point}: by~\eqref{eq:est_Aeps} we have
\begin{align*}
  \| g \|_{p, k_\delta} &\leq 2 \| \psi_\epsilon f \|_{p, k_\delta} 
\end{align*}
thus
\begin{align*}
  \| Ef \|_{p, \tilde{P}_0 k_\delta} = \| F_0 * g \|_{p, \tilde{P}_0 k_\delta} \leq \| F_0 \|_{\infty, \tilde{P}_0} \| g \|_{p, k_\delta} \leq 2 \| F_0 \|_{\infty, \tilde{P}_0} \| \psi_\epsilon f \|_{p, k_\delta} \leq 4 \| F_0 \|_{\infty, \tilde{P}_0} \| \psi_\epsilon \|_{1,1} \| f \|_{p, k_\delta}
\end{align*}
where we used Lemma~\ref{lem:kdelta} again. On the other hand, Lemma~\ref{lem:kdelta0} ensures that the norms $\| \cdot \|_{p, k_\delta}$ and $\| \cdot \|_{p, k}$ are equivalent: an explicit calculation actually shows that  
\begin{align*}
  \| u \|_{p, k} \leq \| u \|_{p, k_\delta} \leq C_\delta \| u \|_{p, k}, \quad \forall u \in \mathcal{B}_{p, k}.
\end{align*}
In the same manner one obtains
\begin{align*}
  \| u \|_{p, \tilde{P}_0 k} \leq \| u \|_{p, \tilde{P}_0 k_\delta} \leq C_\delta \| u \|_{p, \tilde{P}_0 k}, \quad \forall u \in \mathcal{B}_{p, k}
\end{align*}
so now we have
\begin{align*}
  \| Ef \|_{p, \tilde{P}_0 k} \leq \| Ef \|_{p, \tilde{P}_0 k_\delta} \leq 4 \| F_0 \|_{\infty, \tilde{P}_0} \| \psi_\epsilon \|_{1,1} \| f \|_{p, k_\delta} \leq 4 C_\delta \| F_0 \|_{\infty, \tilde{P}_0} \| \psi_\epsilon \|_{1,1} \| f \|_{p, k}.
\end{align*}
\qed

\section{Proof of Proposition~9.15}

In this section we follow very closely the arguments in~\cite[pp.~53--56]{ct91}; this is indeed the ``Gevrey version'' of them. Again, the reader is referred to our main article for the statement of Proposition~9.15, which we shall not recall here.

We assume that $g$ and $u$ are such that $\supp \dd g \sset U_0^-$ and $\supp \dd u \sset U_0^+ \cap V_0$: the other case (i.e.~the opposite choice of signs) can be treated analogously. First of all, compactness of $\overline{U}$ ensures the existence of a constant $A > 0$ (which does not depend on $x_0$) such that
  \begin{align*}
    |\Phi(x,t) - \Phi(x_0,t)| &\leq A |x - x_0|, \quad \forall (x,t) \in U.
  \end{align*}
  
  Fix some $\phi \in G^\sigma(\C)$ and define
  \begin{align}
    \phi^\sharp \dfn Z^*\phi = \phi \circ Z \label{eq:dfn_phisharp} 
  \end{align}
  which belongs, for instance, to $G^\sigma(U)$ since $Z$ is a real-analytic map. Denoting by
  \begin{align*}
    \TR{\pi}{(x,t)}{\R \times \R^n}{t}{\R^n}
  \end{align*}
  the projection onto the $t$-variable, we have $U_0 = \pi(U)$ since $U$ is cylindrical, and so 
  \begin{align*}
    \pi^* g &\in G^\sigma(U; \wedge^{q - 1} \C T^* \R^{n + 1}).
  \end{align*}
  This observation allows us to define
  \begin{align}
    F &\dfn \phi^\sharp \wedge \dd \bar{Z} \wedge \pi^* g \label{eq:dfn_F}
  \end{align}
  which belongs to $G^\sigma(U; \wedge^q \C T^* \R^{n + 1})$ and, recalling that over $\Omega$ we have an identification $\wedge^{q} \C T^* \R^{n + 1} \cong \Lambda^{0, q}  \oplus \T'^{1, q - 1}$ we can define $f \in G^\sigma(U; \Lambda^{0, q})$ as the (unique) component of $F$ in that direct sum. We claim that if the support of $\phi$ is conveniently chosen we can achieve $\dd' f = 0$ i.e.~$\dd F$ will be a section of $\T'^{1, q}$. Indeed, without extra assumptions we have
  \begin{align*}
    \dd F &= \dd \left( \phi^\sharp \wedge \dd \bar{Z} \right) \wedge \pi^* g - \phi^\sharp \wedge \dd \bar{Z} \wedge \dd \left( \pi^* g \right)\\
    &= \dd \phi^\sharp \wedge \dd \bar{Z} \wedge \pi^* g - \phi^\sharp \wedge \dd \bar{Z} \wedge \pi^* (\dd g).
  \end{align*}
  However
  \begin{align*}
    \dd \phi^\sharp = \dd (Z^* \phi) = Z^* (\dd \phi) = Z^* \left( \frac{\del \phi}{\del z} \wedge \dd z + \frac{\del \phi}{\del \bar{z}} \wedge \dd \bar{z} \right) = \left(  \frac{\del \phi}{\del z} \circ Z \right) \wedge \dd Z +  \left(  \frac{\del \phi}{\del \bar{z}} \circ Z \right) \wedge \dd \bar{Z} 
  \end{align*}
  hence
  \begin{align*}
    \dd \phi^\sharp \wedge \dd \bar{Z} \wedge \pi^* g &= \left(  \frac{\del \phi}{\del z} \circ Z \right) \wedge \dd Z \wedge \dd \bar{Z} \wedge \pi^* g
  \end{align*}
  is a section of $\T'^{1, q}$ over $U$: if we can prove that $\phi^\sharp \wedge \dd \bar{Z} \wedge \pi^* (\dd g)$ is also a section of $\T'^{1, q}$ then our claim will follow. This is where the choice of $\phi$ (or, rather, its support) kicks in: we can choose it so that this summand is actually zero.

  Indeed, let $a > 0$ and $b \in \R$ and define the strip
  \begin{align*}
    E(a,b) &\dfn \{ x + i y \in \C \st |x - x_0| \leq a, \ y \geq b \}.
  \end{align*}
  From the definition of $U_0^-$ we have
  \begin{align*}
    \pi^{-1} (U_0^{-}) &= \{ (x,t) \in U \st \Phi(x_0,t) < y_0 \} \\
    Z^{-1}(E(a,b)) &= \{(x,t) \in U \st |x - x_0| \leq a, \ \Phi(x,t) \geq b \}
  \end{align*}
  and if $(x,t) \in Z^{-1}(E(a,b)) \cap \pi^{-1}(U_0^-)$ then
  \begin{align*}
    b \leq \Phi(x,t) \leq |\Phi(x,t) - \Phi(x_0,t)| + \Phi(x_0,t) < A|x - x_0| + y_0 \leq Aa + y_0.
  \end{align*}
  So if we choose $a,b$ such that $y_0 + Aa \leq b$ then $Z^{-1}(E(a,b)) \cap \pi^{-1}(U_0^-) = \eset$. In particular, choosing $\supp \phi \sset E(a,b)$ yields
  \begin{align*}
    \supp \phi^\sharp = \supp Z^* \phi = Z^{-1}(\supp \phi) \sset Z^{-1}(E(a,b)).
  \end{align*}
  Since we already had
  \begin{align*}
    \supp \pi^* (\dd g) = \pi^{-1}(\supp \dd g) \sset \pi^{-1}(U_0^-)
  \end{align*}
  for $\supp \dd g \sset U_0^-$ by hypothesis, we must have $\supp \phi^\sharp$ and $\supp \pi^* (\dd g)$ disjoint, hence $\phi^\sharp \wedge \dd \bar{Z} \wedge \pi^* (\dd g)$ vanishes in $U$. We conclude that $\dd' f = 0$.

  We introduce a new parameter $r > 0$ (to be specified later) and let $\chi \in G^\sigma_c(\R)$ be such that $0 \leq \chi \leq 1$ and
  \begin{align*}
    \chi(x) &= 1 \ \text{if $|x - x_0| < r/2$},\\
    \chi(x) &= 0 \ \text{if $|x - x_0| > r$}.
  \end{align*}
  Let also $\tilde{\chi} \in G^\sigma(\R \times \R^n)$ be defined as
  \begin{align*}
    \tilde{\chi}(x,t) &\dfn \chi(x), \quad (x,t) \in \R \times \R^n,
  \end{align*}
  hence
  \begin{align*}
    v &\dfn \tilde{\chi} \wedge \dd Z \wedge \pi^* u
  \end{align*}
  is a section of $\Lambda^{1,n - q}$ with $G^\sigma$ coefficients. Since $\supp u \sset V_0$ we have that
  \begin{align*}
    \supp v \sset \supp \tilde{\chi} \cap \pi^{-1}(\supp u) \sset \{(x,t) \in V \st |x - x_0| \leq r, \ t \in \supp u \}
  \end{align*}
  the latter being a compact subset of $V$ if we choose $r > 0$ sufficiently small: in that case $v \in G^\sigma_c(V;\Lambda^{1,n - q})$. It follows from all the definitions that 
  \begin{align*}
    f \wedge v &= F \wedge v = \phi^\sharp \wedge \dd \bar{Z} \wedge \pi^* g \wedge \tilde{\chi} \wedge \dd Z \wedge \pi^* u = \pm \left( \tilde{\chi} \ \phi^\sharp \right) \wedge \dd Z \wedge \dd \bar{Z} \wedge \pi^*(g \wedge u).
  \end{align*}
  We remark that the first identity follows from the fact that $f - F$ is a section of $\T'^{1, q - 1}$ (so its wedge with $v$ is zero) and that the correct sign in the last identity is irrelevant for our purposes: we are only interested in studying the vanishing of their integrals. Also, recalling that $\supp \phi^\sharp \sset Z^{-1}(E(a,b))$ and that $\tilde{\chi}{(x,t)} = 1$ if $|x - x_0| < r/2$, it is clear that if we further impose that $a < r/2$ then $\tilde{\chi} = 1$ on $\supp \phi^\sharp$, and hence
  \begin{align*}
    f \wedge v &= \pm \ \phi^\sharp \wedge \dd Z \wedge \dd \bar{Z} \wedge \pi^*(g \wedge u).
  \end{align*}
  
  Now notice that
  \begin{align*}
    \phi^\sharp \wedge \dd Z \wedge \dd \bar{Z} = (Z^* \phi) \wedge \dd Z \wedge \dd \bar{Z} = Z^* \left( \phi \wedge \dd z \wedge \dd \bar{z} \right) = 2i \ Z^* \left( \phi \wedge \dd y \wedge \dd x \right).
  \end{align*}
  We will now assume that $\phi$ is non-negative, and define $\psi_0 \in G^\sigma(\C; \R)$ as
  \begin{align*}
    \psi_0(x + iy) &\dfn \int_{-\infty}^y \phi(x + is) \ \dd s
  \end{align*}
  which clearly satisfies
  \begin{align*}
    \frac{\del \psi_0}{\del y} &= \phi.
  \end{align*}
  A simple calculation also shows that since $E(a,b)$ contains $\supp \phi$ then it also contains $\supp \psi_0$. Letting $\psi \dfn \psi_0 \wedge \dd x \in G^\sigma(\C; \wedge^1 T^* \C)$ we conclude that
  \begin{align*}
    \dd \psi = \dd \psi_0 \wedge \dd x = \frac{\del \psi_0}{\del y} \wedge \dd y \wedge \dd x = \phi \wedge \dd y \wedge \dd x
  \end{align*}
  hence
  \begin{align*}
    f \wedge v &= \pm \ 2i \ Z^* \left( \phi \wedge \dd y \wedge \dd x \right) \wedge \pi^*(g \wedge u)\\
    &= \pm \ 2i \ Z^* \left( \dd \psi \right) \wedge \pi^*(g \wedge u)\\
    &= \pm \ 2i \ \dd (Z^* \psi) \wedge \pi^*(g \wedge u).
  \end{align*}

  We claim that, for the choices above, $Z^* \psi \wedge \pi^*(g \wedge u)$ is compactly supported in $U$. Indeed, since $\supp \psi = \supp \psi_0 \sset E(a,b)$ we have $\supp Z^* \psi \sset Z^{-1}(E(a,b))$ and thus
  \begin{align*}
    (\supp Z^* \psi) \cap (\supp \pi^* u) &\sset \{ (x, t) \in U \st |x - x_0| \leq a, \ t \in \supp u \}
  \end{align*}
  the latter a compact subset of $U$, while the former clearly contains the support of $Z^* \psi \wedge \pi^*(g \wedge u)$, hence our claim. It then follows from Stokes's Theorem that 
  \begin{align*}
    0 = \int \dd \left( Z^* \psi \wedge \pi^*(g \wedge u) \right) =  \int \dd (Z^* \psi) \wedge \pi^*(g \wedge u) \pm \int Z^* \psi \wedge \dd \pi^*(g \wedge u) 
  \end{align*}
  which in turn implies
  \begin{align*}
    \int f \wedge v &= \pm \ 2i \int Z^* \psi \wedge \dd \pi^*(g \wedge u).
  \end{align*}
  But notice that
  \begin{align*}
    Z^* \psi \wedge \dd \pi^*(g \wedge u) &= Z^* \psi \wedge \pi^*(\dd g \wedge u) \pm Z^* \psi \wedge \pi^*(g \wedge \dd u)
  \end{align*}
  where the first summand is zero since $\supp (Z^* \psi) \cap \supp \pi^*(\dd g) = \eset$: this follows from the fact that $\supp \psi \sset E(a,b)$ and $\supp \dd g \sset U_0^-$, and thus implies that
  \begin{align*}
    \int f \wedge v &= \pm \ 2i \int Z^* \psi \wedge \pi^*(g \wedge \dd u).
  \end{align*}

  Now we are going to impose further restrictions on $\phi$. Recall that $\supp \dd u \sset U_0^+ \cap V_0$, meaning that $\Phi(x_0,t) > y_0$ for all $t \in \supp \dd u$: by compactness, there exists $\rho > 0$ such that
  \begin{align*}
    \Phi(x_0, t) &> y_0 + \rho, \quad \forall t \in \supp \dd u.
  \end{align*}
  Once again we shrink $r > 0$ so that $2Ar < \rho$, and thus $y_0 + Ar < -Ar + y_0 + \rho$, which allows us to choose $b,b',b'' \in \R$ such that
  \begin{align*}
    y_0 + Ar \leq b < b' < b'' < -Ar + y_0 + \rho.
  \end{align*}
  If we further assume that
  \begin{align*}
    \supp \phi &\sset \{ x + iy \st |x - x_0| \leq a, \ b \leq y \leq b' \}
  \end{align*}
  then it follows from the definition of $\psi_0$ that 
  \begin{align*}
    y > b' &\impl \psi_0(x + iy) = \psi_0(x + ib'), \quad \forall x \in \R.
  \end{align*}
  For $|x - x_0| \leq a$ and $t \in \supp \dd u$ we then have
  \begin{align*}
    \Phi(x,t) &= \left( \Phi(x,t) - \Phi(x_0,t) \right) + \Phi(x_0,t)\\
    &\geq -A |x - x_0| + \Phi(x_0,t) \\
    &\geq -Aa + y_0 + \rho\\
    &> -Ar + y_0 + \rho\\
    &> b'
  \end{align*}
  which implies that
  \begin{align*}
    \psi_0(Z(x,t)) = \psi_0(x + i \Phi(x,t)) = \psi_0(x + ib')
  \end{align*}
  holds whenever $|x - x_0| \leq a$ and $t \in \supp \dd u$.

  Now recall that $U$ is a cylindrical open set centered at the origin, hence there exists an open interval $I \sset \R$ centered at $0$ such that $U = I \times U_0$. Hence
  \begin{align*}
    C(x) \dfn \psi_0(x + ib') = \int_{-\infty}^{b'} \phi (x + is) \dd s
  \end{align*}
  defines a function $C:I \rarr \R$ which allows us to write
  \begin{align*}
    Z^* \psi = Z^* (\psi_0 \wedge \dd x) = (\psi_0 \circ Z) \wedge \dd (x \circ Z) = C(x) \wedge \dd x
  \end{align*}
  for $(x,t) \in U$ such that $|x - x_0| \leq a$ and $t \in \supp \dd u$. It is also clear that
  \begin{align*}
    \supp \left( Z^* \psi \wedge \pi^*(g \wedge \dd u) \right) &\sset \{ (x,t) \in U \st |x - x_0| \leq a, \ t \in \supp \dd u \}
  \end{align*}
  and therefore
  \begin{align*}
    \int f \wedge v = \pm \ 2i \int Z^* \psi \wedge \pi^*(g \wedge \dd u) = \pm \ 2i \int C(x) \wedge \dd x \wedge \pi^*(g \wedge \dd u) = \pm \ 2i \left( \int C(x) \ \dd x \right) \int g \wedge \dd u
  \end{align*}
  where
  \begin{align*}
    \int C(x) \ \dd x = \int \int_{-\infty}^{b'} \phi (x + is) \ \dd s \ \dd x = \int_\C \phi \neq 0
  \end{align*}
  if we assume $\phi$ nonzero: equivalence~(9.5) from~\cite{araujo18} is proven.

  We now turn to the second part of the statement: we will prove that if we shrink $a > 0$ as well as the difference $b' - b > 0$ (but keeping $b$ fixed) then there exists $H \in \OO(\C)$ such that
  \begin{align*}
    \text{$\Re H \leq 0$ in $Z(\supp f)$}, & \quad \text{$\Re H > 0$ in $Z(\supp \dd' v)$}.
  \end{align*}
  Recall that $\supp f \sset \supp F$, and from~\eqref{eq:dfn_F} and~\eqref{eq:dfn_phisharp} we have
  \begin{align*}
    \supp F \sset \supp \phi^\sharp \cap \supp \pi^* g = Z^{-1}(\supp \phi) \cap \pi^{-1}(\supp g)\\
  \end{align*}
  and thus
  \begin{align*}
    Z(\supp F) \sset \supp \phi \cap Z(\pi^{-1}(\supp g)) \sset \supp \phi \sset \{ x + iy \in \C \st |x - x_0| \leq a, \ b \leq y \leq b' \}.
  \end{align*}
  We denote by $\mathcal{R}$ the latter set above, and also define the quantities
  \begin{align*}
    M &\dfn \max \left\{ \Phi(x,t) \st |x - x_0| \leq r, \ t \in \supp \dd u \right\} \\
    M_+ &\dfn \max \left\{ \Phi(x,t) \st \frac{r}{2} \leq |x - x_0| \leq r, \ t \in \supp u \right\} \\
    M_- &\dfn \min \left\{ \Phi(x,t) \st \frac{r}{2} \leq |x - x_0| \leq r, \ t \in \supp u \right\} 
  \end{align*}
  as well as the following subsets of the complex plane
  \begin{align*}
    \mathcal{A} &\dfn \left\{ x + iy \in \C \st |x - x_0| \leq r, \ b'' \leq y \leq M \right\}\\
    \mathcal{B} &\dfn \left\{ x + iy \in \C \st \frac{r}{2} \leq |x - x_0| \leq r, \ M_- \leq y \leq M_+ \right\}.
  \end{align*}
  We claim that $Z(\supp \dd' v) \sset \mathcal{A} \cup \mathcal{B}$. In order to check this, notice first that since $v$ is a section of $\Lambda^{1, n - q}$ we have
  \begin{align*}
    \dd' v = \dd v  \dd \left( \tilde{\chi} \wedge \dd Z \wedge \pi^* u \right) = \dd \tilde{\chi} \wedge \dd Z \wedge \pi^* u - \tilde{\chi} \wedge \dd Z \wedge \pi^* (\dd u)
  \end{align*}
  hence, clearly,
  \begin{align*}
    \supp \dd' v &\sset \left( \supp \dd \tilde{\chi} \cap \supp \pi^* u \right) \cup \left( \supp \tilde{\chi} \cap \supp \pi^* (\dd u) \right).
  \end{align*}
  On the other hand
  \begin{align*}
    \supp \tilde{\chi} &\sset \{ (x,t) \in \R \times \R^n \st |x - x_0| \leq r \}\\
    \supp \dd \tilde{\chi} &\sset \left\{ (x,t) \in \R \times \R^n \st \frac{r}{2} \leq |x - x_0| \leq r \right\}\\
    \supp \pi^* u &\sset \{ (x,t) \in U \st t \in \supp u \}\\
    \supp \pi^*(\dd u) &\sset \{ (x,t) \in U \st t \in \supp \dd u \}
  \end{align*}
  which, together, ensure that $\supp \dd' v$ is contained in the union of the sets below:
  \begin{align*}
     S_1 &\dfn \left \{ (x,t) \in U \st \frac{r}{2} \leq |x - x_0| \leq r, \ t \in \supp u \right\} \\
     S_2 &\dfn \{ (x,t) \in U \st |x - x_0| \leq r, \ t \in \supp \dd u \}.
  \end{align*}
  Clearly, $Z$ maps $S_1$ into $\mathcal{B}$. Also, if $(x,t) \in S_2$ we have
  \begin{align*}
    \Phi(x,t) &= \left( \Phi(x,t) - \Phi(x_0,t) \right) + \Phi(x_0,t)\\
    &\geq -A |x - x_0| + \Phi(x_0,t) \\
    &> -Ar + y_0 + \rho\\
    &> b''
  \end{align*}
  and from the definitions of $Z$, $M$, $S_2$ and $\mathcal{A}$ we have $Z(x,t) \in \mathcal{A}$, proving our claim. 
  
  \begin{figure}
    \begin{center}
      \includegraphics[scale=0.3]{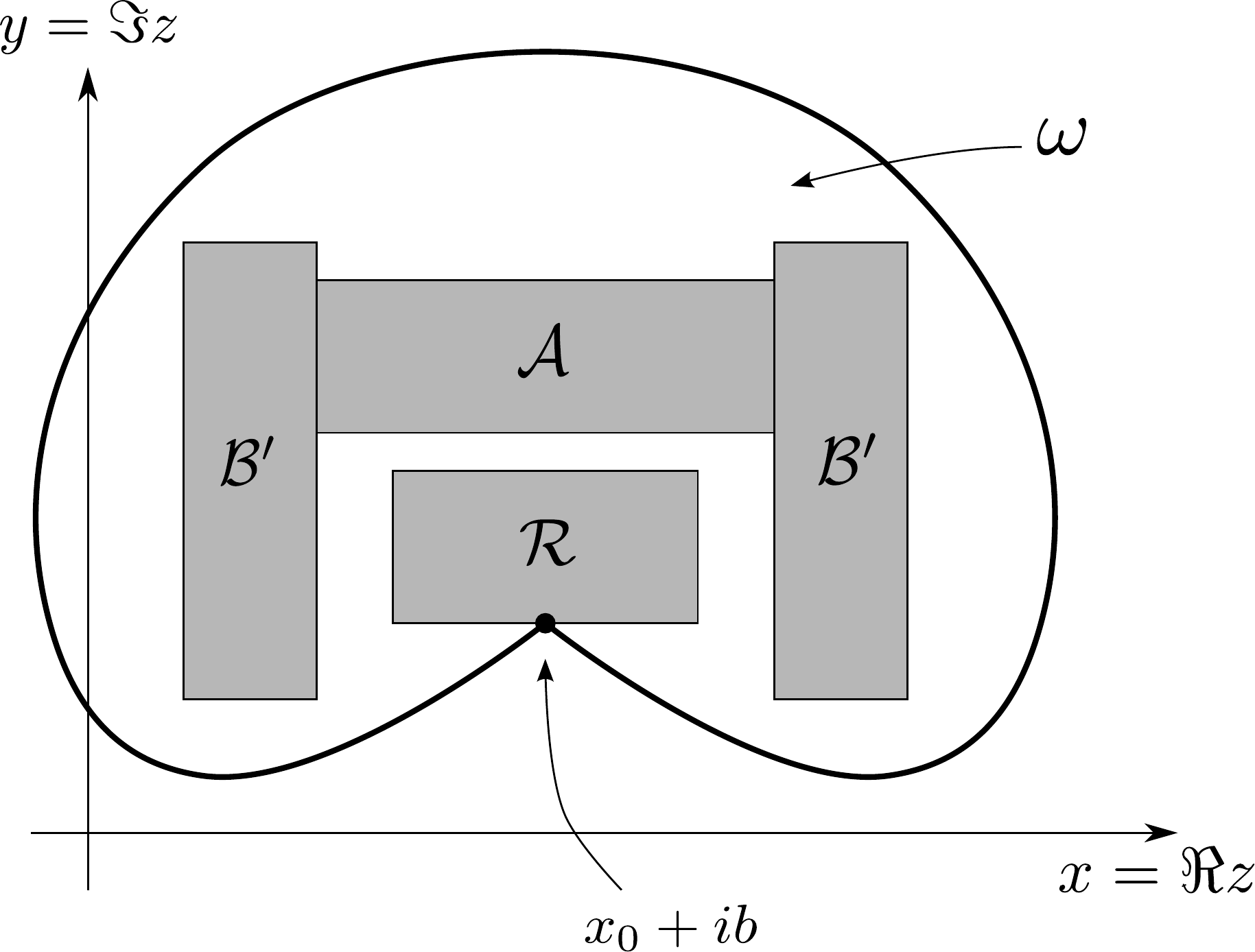}
    \end{center}
    \caption{The compact sets $\mathcal{H} = \mathcal{A} \cup \mathcal{B}'$ and $\mathcal{R}$, which are disjoint; and the open set $\omega$, which contains both of them.}
    \label{fig:horseshoe}
  \end{figure}
  
  \begin{figure}
    \begin{center}
      \includegraphics[scale=0.3]{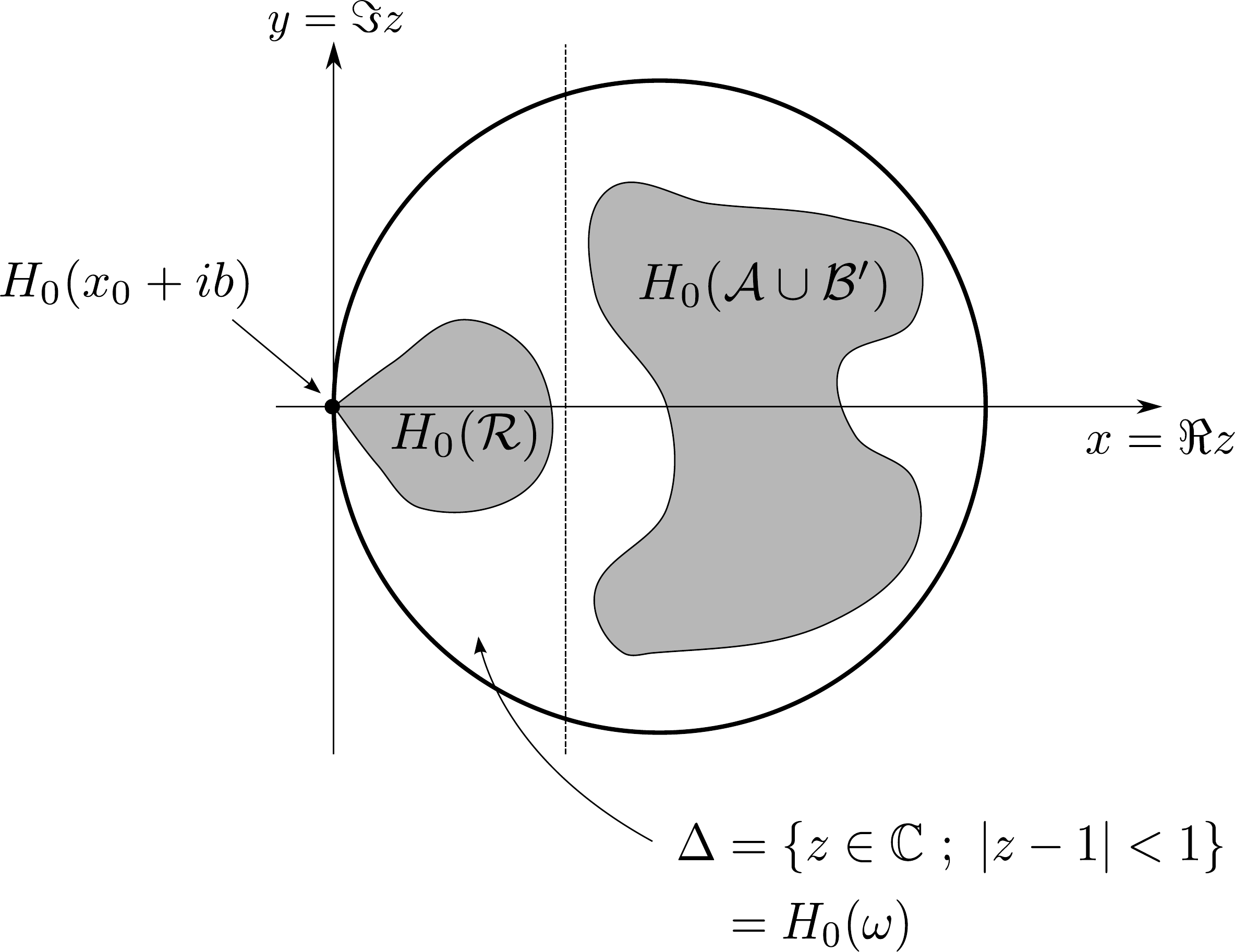}
    \end{center}
    \caption{The scheme presented in Figure~\ref{fig:horseshoe}, now deformed by the homeomorphism $H_0$.}
    \label{fig:horseshoe_deformed}
  \end{figure}

  For a better visualization of the argument, we define the sets 
  \begin{align*}
    \mathcal{B}' &\dfn \left\{ x + iy \in \C \st \frac{r}{2} \leq |x - x_0| \leq r, \ \min \{ M_-,b \} \leq y \leq \max \{ M_+, M \} \right\}
  \end{align*}
  (which contains $\mathcal{B}$) and $\mathcal{H} \dfn \mathcal{A} \cup \mathcal{B}'$ which, on the one hand, contains $\mathcal{A} \cup \mathcal{B}$, and, on the other hand, does not intercept $\mathcal{R}$ (see Figure~\ref{fig:horseshoe}). It is clear that there exists a bounded open set $\omega \sset \C$, which is connected and simply connected, such that:
  \begin{enumerate}
  \item it contains $\mathcal{A} \cup \mathcal{B}$ and $\mathcal{R}$, except for the point $x_0 + ib \in \del \mathcal{R}$;
  \item its boundary is a Jordan curve that contains the point $x_0 + ib$; and
  \item $\C \setminus \overline{\omega}$ is connected.
  \end{enumerate}
  Let $\Delta \sset \C$ stand for the unit open disc centered at $1$: a result due to Carath{\'e}odory ensures the existence of a homeomorphism $H_0: \overline{\omega} \rarr \overline{\Delta}$ which is a biholomorphism between interiors, and we can assume without loss of generality that $H_0(x_0 + ib) = 0$ (see Figure~\ref{fig:horseshoe_deformed}). In particular, $\Re H_0(z) > 0$ for every $z \in \overline{\omega}$ except for $z = x_0 + ib$. Since $\mathcal{H} \sset \omega$ is a compact set, there exists $c > 0$ such that 
  \begin{align*}
    \Re H_0 &> 2c \ \text{in $\mathcal{H}$}.
  \end{align*}
  Also, if we further shrink $a$ and choose $b'$ sufficiently close to $b$ (so that $\mathcal{R}$ is ``thin'' in the $y$-direction) then 
  \begin{align*}
    \Re H_0 &< \frac{c}{4} \ \text{in $\mathcal{R}$}.
  \end{align*}
  
  Finally, Mergelyan's Theorem allows us to approximate $H_0$ by an entire function $H_1$ such that
  \begin{align*}
    \Re H_1 &> \frac{3c}{2} \ \text{in $\mathcal{H}$} \\
    \Re H_1 &< \frac{c}{2} \ \text{in $\mathcal{R}$}
  \end{align*}
  thus setting $H \dfn H_1 - c$ finishes the proof. \qed

\def\cprime{$'$}

\end{document}